\documentclass[a4paper,11pt]{amsart}
\usepackage{graphicx} 

\title{A note on the modal logic of symmetric extensions}
\author{Hope Duncan}
\date{6th May 2026}
\thanks{No data are associated with this article. For the
purpose of open access, the author has applied a Creative Commons Attribution (CC-BY) licence
to any Author Accepted Manuscript version arising from this submission.}

\usepackage[utf8]{inputenc} 

\usepackage{lmodern} 
\usepackage[T1]{fontenc} 

\usepackage{amsfonts} 
\usepackage{amsmath}
\usepackage{amssymb}
\usepackage{amsthm}
\usepackage{ dsfont } 
\usepackage{tikz-cd}
\usepackage[T1]{fontenc}

\usepackage[bottom]{footmisc} 
\usepackage{graphicx} 
\usepackage{mathrsfs}

\newtheorem{theorem}{Theorem}[section]
\newtheorem{prop}[theorem]{Proposition}
\newtheorem{cor}[theorem]{Corollary}
\newtheorem{lemma}[theorem]{Lemma}

\newtheorem{question}{Question}
\newtheorem{remark}[theorem]{Remark}

\newtheorem*{main}{Main Theorem}

\theoremstyle{definition}
\newtheorem{definition}[theorem]{Definition}

\newcommand{\Q}{\mathbb{Q}} 
\newcommand{\p}{\mathbb{P}}

\newcommand{\AC}{\mathsf{AC}}
\newcommand{\ZF}{\mathsf{ZF}}
\newcommand{\ZFC}{\mathsf{ZFC}}
\newcommand{\HS}{\mathsf{HS}}

\newcommand{\SVC}{\mathsf{SVC}}
\newcommand{\CH}{\mathsf{CH}}
\newcommand{\GCH}{\mathsf{GCH}}

\newcommand{\EC}{\mathsf{EC}}

\newcommand{\Sf}{\mathsf{S}4.2}

\usepackage{hyperref} 
\hypersetup{
    colorlinks=true,
    citecolor=blue,
    linkcolor=blue}

\begin{document}

\begin{abstract}

    Taking symmetric extensions can be considered as a generalisation of forcing, which produces a richer multiverse of models with and without the axiom of choice. We can study the structure of this multiverse using modal logic. In particular, we define the concept of of choice-switches, and show any independent system of choice-switches is not itself independent from any standard example of an independent system of buttons.
\end{abstract}

\maketitle

\section{Introduction}

A standard way of generating models of $\ZF$ without choice is by taking a symmetric extension over a model of $\ZFC$. Symmetric extensions can be considered to be a generalisation of forcing, where forcing is thought of as taking a symmetric extension with names symmetric for the improper subgroup filter (which contains the trivial group). In a similar way we can think of the multiverse of symmetric extensions as a generalisation of the forcing multiverse, adding intermediate models generated by symmetric extensions which may witness failures of choice. Modal logic gives us a lens through which to study the structure of this multiverse. Modal logic adds two modal operators $\square$ and $\lozenge$, where $\square \varphi$ expresses that $\varphi$ is \emph{necessary}, and $\lozenge \varphi$ that $\varphi$ is \emph{possible}. In the context of a forcing multiverse, $\varphi$ being necessary corresponds to $\varphi$ being true in every forcing extension, and $\varphi$ being possible to $\varphi$ being true in some forcing extension. 

There are various axiom systems for modal logic we can define, the most important for us being $\Sf$ (see Definition \ref{s4.2:def}). Hamkins and Löwe showed that the modal logic of forcing is $\Sf$ \cite{Hamkins-modal-logic}, and using a similar proof Block and Löwe showed that the modal logic of symmetric extensions is also $\Sf$. A key part of the proof that the modal logic of forcing is S4.2 is the existence of an infinite independent system of buttons and switches. A \emph{button} is a statement $\varphi$ which is (necessarily) possibly necessary, and a \emph{switch} is a statement $\varphi$ such that both $\varphi$ and $\neg\varphi$ are necessarily possible. Buttons and switches are \emph{independent} if they can be pushed/switched without affecting any other button or switch. 

In this generalised context, it is natural to consider whether we now have new buttons and switches. We show there are no new buttons, but there are new switches. As being a symmetric extension (of a model of $\ZFC$) is equivalent to choice being forceable, being a symmetric extension of a model of $\ZFC$ is in some sense characterised by the axiom of choice being a switch. We will consider how generalising to symmetric extensions affects these buttons and switches. In particular, we work with \emph{choice-switches}, switches which alter which sets are well-orderable (see Definition \ref{choice-switch:def} for a precise statement). 

\begin{main}
    Any independent system of choice-switches is not independent from any known independent system of buttons, that is, buttons which assert some combinatorial property of a regular cardinal holds. 
\end{main}

In section 2 we will give a brief overview of how symmetric extensions are constructed, and recall some standard results about clubs. In section 3 we will review Karagila's model in which Fodor's Lemma fails, which is used to show that consistently, well-ordering the universe will collapse cardinals. In section 4 we will recall important definitions relating to the modal logic of forcing, and prove our Main Theorem, that a system of choice-switches cannot be independent from any of our current examples of buttons. Finally, we present some open questions for future inquiry, notably into models of $\ZF + \neg \SVC$. The appendix summarises work of Ryan-Smith, which can be used to construct an independent system of choice-switches.  

\section{Preliminaries}

\subsection{Symmetric systems}
In the following section, assume that $V$ is a model of $\ZF$. Let $\p$ be a forcing notion and consider an automorphism $\pi \in \text{Aut}(\p)$. We can extend the action of $\pi$ to $\p$ names with the following recursive definition, \[\pi \dot{x} = \{ \langle \pi p, \pi \dot{y}\rangle \mid \langle p, \dot{y} \rangle \in \dot{x} \}.\]
Let $\mathscr{G}$ be a group, we say that $\mathscr{F}$ is a normal filter of subgroups of $\mathscr{G}$ if it is a non-empty family of subgroups of $\mathscr{G}$ closed under finite intersections and supergroups, such that for every $\pi \in \mathscr{G}$ and $H \in \mathscr{F}$, $\pi H \pi^{-1} \in \mathscr{F}$. 

\begin{definition}
    A \emph{symmetric system} $\mathscr{S}$ is a triple $\langle \p, \mathscr{G}, \mathscr{F} \rangle$, where $\p$ is a forcing notion, $\mathscr{G} \subseteq \text{Aut}(\p)$ is a group of automorphisms, and $\mathscr{F}$ is a normal filter of subgroups on $\mathscr{G}$.
\end{definition}

Let $\langle \p, \mathscr{G}, \mathscr{F} \rangle$ be a symmetric system and let $\dot{x}$ be a $\p$-name. We say that $\dot{x}$ is $\mathscr{F}$-symmetric if sym$_\mathscr{G}(\dot{x}) = \{ \pi \in \mathscr{G} \mid \pi \dot{x} = \dot{x}\} \in \mathscr{F}$. If this property holds hereditarily, i.e. all the names which appear in $\dot{x}$ are also $\mathscr{F}$-symmetric, we say that $\dot{x}$ is hereditarily $\mathscr{F}$-symmetric. We will use $\HS_\mathscr{F}$ to denote the class of hereditarily $\mathscr{F}$-symmetric names, often omitting the subscript when this is clear from context. 

\begin{theorem}
    Suppose that $\langle \p, \mathscr{G}, \mathscr{F} \rangle$ is a symmetric system, and $G \subseteq \p$ is a $V$-generic filter. Then $V_\mathscr{S} = \{ \dot{x}^G \mid \dot{x} \in \HS\}$ is a transitive subclass of $V[G]$ which contains $V$ and satisfies the axioms of $\ZF$. 
\end{theorem}

We refer to this $V_\mathscr{S}$ as a symmetric extension of $V$. We can consider taking symmetric extensions simply as a generalisation of taking forcing extensions, we recover forcing by taking a symmetric extension with names symmetric for the improper filter (which contains the trivial group).

We also know exactly `how far' away from a model of choice a symmetric extension is, a symmetric extension is always a model of small violations of choice, henceforth referred to as $\SVC$. 

\begin{definition}
    We say that \emph{small violations of choice} holds in a model if there is a set $S$ such that for all $X$ there is an ordinal $\eta$ and a surjection $f: S \times \eta \rightarrow X$.
\end{definition}

Results of Blass \cite{Blass-injectivity} and Usuba \cite{usuba} give us a characterisation of being a symmetric extension. It is important for our purposes that we can always force to restore choice in a model of $\SVC$. 

\begin{theorem}\label{svc-equivalences}
    The following are equivalent: 
    \begin{itemize}
        \item $M \models \SVC$
        \item There is a notion of forcing $\p \in M$ such that $\mathds{1}_\p \Vdash \AC$
        \item There is an inner model $V \subseteq M$ such that $V \models \ZFC$, and there is a symmetric system $\mathscr{S} = \langle \p, \mathscr{G}, \mathscr{F} \rangle$ such that $M = \HS^G_\mathscr{S}$ for some $V$-generic $G \subseteq \p$
        \item There is an inner model $V \subseteq M$ such that $V \models \ZFC$ and there is $x \in M$ such that $M = V(x)$
        
    \end{itemize}
\end{theorem}

\subsection{Clubs and collapses}

We will rely on the following standard facts about clubs.

\begin{prop}
    Suppose that $\lambda$ is an uncountable regular cardinal, $\{S_\alpha \mid \alpha < \mu\}$ is a sequence of clubs of $\lambda$. If $\mu < \lambda$, then $\bigcap\{S_\alpha \mid \alpha < \mu\}$ is a club; and if $\mu = \lambda$ then the diagonal intersection $\triangle\{S_\alpha \mid \alpha < \mu\}$ is a club.
\end{prop}

\begin{prop}\label{collapse}
    Let $\kappa$ be a regular cardinal in $M$ (a model of $\ZF$). If in some extension, $M[G]$, there exists a sequence of $<\kappa$ many clubs, with empty intersection, then $\kappa$ is not regular in $M[G]$. In particular, if $\kappa$ is a successor cardinal in $M$, and $M[G] \models \ZFC$, then $\kappa$ must have been collapsed. 
\end{prop}

\section{Collapsing cardinals}

\subsection{The Fodor property}
Using a construction of Karagila \cite{karagila-fodor}, we know that it is consistent that Fodor's Lemma can fail everywhere. In particular, for any uncountable regular cardinal $\lambda$, and $\mu < \lambda$, there exists a symmetric extension in which the $(\lambda, \mu)$-Fodor property fails. For our purposes, we will assume $\lambda$ is a limit ordinal. 

\begin{definition}
    Let $\lambda$ be an uncountable, regular cardinal, and $\mu \leqslant \lambda$. We say that $\lambda$ has the \emph{$(\mu, \lambda)$-Fodor property} if whenever $S \subseteq \lambda$ is a stationary set, and $f: S \rightarrow \lambda$ is a regressive function with range of cardinality $\leqslant \mu$, there is some $\alpha$ such that $f^{-1}(\{ \alpha\})$ is stationary. 
\end{definition}

We can phrase the Fodor property in terms of a property of the club filter.

\begin{theorem}
    Let $\lambda$ be an uncountable regular cardinal. The following are equivalent for any $\mu < \lambda$: 
    \begin{enumerate}
        \item $\lambda$ has the $(\lambda, \mu)$-Fodor property
        \item The club filter of $\lambda$ is $\mu^+$-complete. 
    \end{enumerate}
\end{theorem}

To show that $(\kappa, \mu)$-Fodor property can fail, we will construct a symmetric extension $M$ in which the club filter is not normal. The symmetric system $\langle \p, \mathscr{F}, \mathscr{G}\rangle$ which generates $M$ is defined as follows: 

\begin{itemize}
    \item $\p$ is an iteration of two forcings, we add a Cohen subset to $\kappa$, which gives us a disjoint sequence of stationary sets $S_\alpha$, then shoot clubs into the complement of each stationary set. 
    \begin{itemize}
        \item Let $\Q_0$ be $\text{Add}(\kappa, \mu)$. 
        \item Denote by $\Q_{1,\alpha}$ the forcing $\text{Club}(S_\alpha)$, and let $\Q_1$ be the $<\mu$ support product $\prod_{\alpha<\mu} \Q_{1,\alpha}$. 
        \item The forcing $\p$ is $\Q_0 \ast \dot{\Q}_1$. 
    \end{itemize}
    
    \item The automorphism group $\mathscr{G}$ is the group of automorphisms induced from names of automorphisms of $\dot{\Q}_1$ which are forced to belong to the $<\mu$-suppor products of the groups Aut($\Q_{1,\alpha}$).  

We say $\sigma \in \text{fix}(E)$ if for every $\alpha \in E$, $\Vdash_{\Q_0} \dot{\sigma} \restriction \dot{\Q}_{1, \alpha} = \text{id}$. 

    \item Define $\mathscr{F}$ to be the filter of subgroups generated by $\text{fix}(E)$, for $E \in [\mu]^{<\mu}$.

\end{itemize}

\begin{theorem}\label{not-fodor}\cite[Theorem 4.5]{karagila-fodor}
    In $M$, for every $\alpha < \mu$, $S_\alpha$ contains a club, but the intersection $\bigcap_{\alpha<\mu}S_\alpha = \varnothing$. Namely, the $(\kappa,\mu)$-Fodor property fails.
\end{theorem}

\subsection{Well-ordering the universe}

First, we want to check that cardinals have not been collapsed in $M$, the symmetric extension. 

\begin{prop}\label{not-collapsed}
    $\kappa$ has not been collapsed in $M$.
\end{prop}

\begin{proof}
    The combination of Lemma \ref{4.2} and Lemma \ref{4.3} shows that all the symmetric names are added by a $\kappa$-closed forcing, and so in $M$, $\kappa$ has not been collapsed. 
\end{proof}

\begin{lemma}\label{4.2}\cite[Lemma 4.2]{karagila-fodor}
    Suppose that $\dot{A} \in \HS$ is a $\p$-name for a set of ordinals. Then there is some $\alpha < \mu$ and $\dot{A}' \in \HS$ which is $\Q_0 \ast \dot{\Q}_1 \restriction \alpha$-name and $\mathds{1} \Vdash \dot{A} = \dot{A}'$. 
\end{lemma}

\begin{lemma}\label{4.3}\cite[Lemma 4.3]{karagila-fodor}
    Suppose that $\alpha < \mu$. Then the two step iteration $\Q_0 \ast \dot{\Q}_1$ has a $\kappa$-closed dense subset. 
\end{lemma} 

Now, let us force the universe to be well-orderable again, and call this new model $N$. Recall by Theorem \ref{svc-equivalences}, we can always do this. 

\begin{prop}\label{collapsed}
    $\kappa$ has been collapsed in $N$. 
\end{prop}

\begin{proof}
    In $M$, each $S_\alpha$ contains a club, since the support $E = \{\alpha\}$ witnesses the generic club for $\Q_{1,\alpha}$ has a symmetric name. Additionally, $\bigcap_{\alpha < \mu} S_\alpha = \varnothing$ by definition. If we  well-ordering the universe, we now have a sequence of $< \kappa$ clubs with empty intersection, so as we picked $\kappa$ to be a limit ordinal, by Proposition \ref{collapse}, $\kappa$ has been collapsed.
\end{proof}

\section{The modal logic of symmetric extensions}

\subsection{The modal logic of forcing}
We can consider all the possible forcing extensions of a model of $\ZFC$, as a multiverse, with forcing acting as an `accessibility relation' between them. To study this multiverse, it can be helpful to use concepts from modal logic.  

Modal logic adds two modal operators $\square$ and $\lozenge$, where $\square \varphi$ expresses that $\varphi$ is necessary, and $\lozenge \varphi$ that $\varphi$ is possible. In the context of a forcing multiverse, $\varphi$ being necessary corresponds to $\varphi$ being true in every forcing extension, and $\varphi$ being possible to $\varphi$ being true in some forcing extension. Hamkins and Löwe gave an exact categorisation of the modal logic of forcing in \cite{Hamkins-modal-logic}. 

\begin{theorem}\label{modal-forcing}
    The modal logic of forcing is exactly $\Sf$. 
\end{theorem}

\begin{definition}\label{s4.2:def}

The axiom system $\Sf$ is as follows: 
\begin{itemize}
        \item (K): $\square (\varphi \rightarrow \psi) \rightarrow (\square \varphi \rightarrow \square\psi)$
        \item (Dual): $\neg \square \varphi \leftrightarrow \square \neg \varphi $
        \item (S): $\square \varphi \rightarrow \varphi $
        \item (4): $\square \varphi \rightarrow \square \square \varphi $
        \item (.2): $\lozenge \square \varphi \rightarrow \square \lozenge \varphi $
\end{itemize}

\end{definition}

The proof of Theorem \ref{modal-forcing} relies on the existence of an infinite independent system of buttons and switches. 

\begin{definition}
    A \emph{button} is a statement  $\varphi$ which is (necessarily)\footnote{Over $\Sf$, necessarily possibly necessary is equivalent to possibly necessary. In weaker systems, this need not be the case.} possibly necessary.
\end{definition}

 The statement $V \neq L$ is a button. The button is \emph{pushed} when it becomes necessary, else it is \emph{unpushed}. So, once the button is pushed, $\varphi$ holds in any further forcing extension. We say a button $\varphi$ is a \emph{pure button} if when $\varphi$ is true, it is necessarily true, i.e. $\varphi \rightarrow \square \varphi$\footnote{For example, $\omega_1^L \neq \omega_1$ is a pure button, however $(\omega_1^L \neq \omega_1 )\lor (\CH$ holds) is a button, but not a pure button.}.  

\begin{definition}
    A \emph{switch} is a statement $\varphi$ such that both $\varphi$ and $\neg\varphi$ are necessarily possible.
\end{definition}

This means we can turn a switch on by moving into a generic extension where it is true, and then turn it off by moving into a generic extension where it is false. For example, $\CH$ is a switch. 

\begin{definition}
    A system of buttons and switches is \emph{independent} over a particular model if (i) all of the buttons are unpushed in that model, and (ii), necessarily, i.e. in any forcing extension, any of the buttons can be pressed, and any of the switches can be switched without affecting any other button or switch. 
\end{definition}

We can exhibit a an independent system of buttons and switches over $V=L$ as follows; take a countable partition of $\omega_1$ into stationary sets $S_n$, and define buttons $b_n := \text{`the set }S_n\text{ is not stationary'}$. This set of buttons is used in \cite{Hamkins-modal-logic}, for independence see \cite[23.8, 23.6]{Jech2003}. Then we can define our switches as $s_n := \text{`}2^{\aleph_n} = \aleph_{n+1} \text{',  for } n \geqslant 2.$

Hamkins and Löwe also observe that $\omega_1^L$ being collapsed is equivalent to the infinite  conjugation of each $b_n$ defined above being pushed. 

\subsection{Choice-switches}

A natural generalisation of this work is to ask what the modal logic of symmetric extensions is. This question was answered by Block and Löwe in \cite{LB}. 

\begin{theorem}\label{modal-symmetric}
    The modal logic of symmetric extensions is exactly $\Sf$. 
\end{theorem}

The proof is similar to the proof of Theorem \ref{modal-forcing}, which if we think about symmetric extensions as a generalisation of forcing extensions is unsurprising. Indeed, the same system of buttons and switches can used in proof of Theorem \ref{modal-symmetric} as in \ref{modal-forcing}. A natural question therefore is if taking symmetric extensions is a generalisation of taking forcing extensions, do we have new examples of pure\footnote{We do not want to consider a trivial addition like `$(\omega_1^L \neq \omega_1)$ or countable choice holds' as a new example of a button.} buttons? Do we have new examples of switches? Quickly, we can answer no, and yes, respectively. For brevity, we will say \emph{forcing switch} for a statement which is a button in the modal logic of forcing, and \emph{symmetric switch} for a statement which is a switch in the modal logic of symmetric extensions. 

\begin{prop}
    Every $\varphi$ which is a pure symmetric button, was already a forcing button.
\end{prop}

\begin{proof}
    If $\varphi$ holds in every symmetric extension of a model, it necessarily holds in every forcing extension. 
\end{proof}

\begin{prop}
    There exist symmetric switches $\varphi$ which were not forcing switches.
\end{prop}

\begin{proof}
    The axiom of choice is now a switch!\footnote{Assuming of course, we are working in a model of $\ZF + \SVC$, but this is the standard setting for studying the modal logic of symmetric extensions. It is unclear what the modal logic of symmetric extensions is over models of $\ZF + \neg\SVC$, see Question \ref{notsvc:question}.}  
\end{proof}

Recall here that being a symmetric extension of a model of $\ZFC$ is equivalent to being able to set force to restore choice. We will now focus on switches which `affect' the axiom of choice. 

\begin{definition}\label{choice-switch:def}
    We say $\varphi$ is a \emph{choice-switch} if necessarily, $\varphi$ and $\neg\varphi$ cannot both hold in a model where the same fragments of the axiom of choice are true. 
\end{definition}

For example, `the axiom of choice holds' is a choice-switch. More interesting examples include `countable choice holds', and `there exists an amorphous set'. However, `$\CH$ holds' is not a choice-switch, we can force $\neg\CH$ by violating choice, but it is not necessary to do so. 

\begin{question}
    Can we construct an infinite independent system of buttons and switches, using choice-switches? 
\end{question}

We have nice examples of independent systems of choice-switches in the literature already. In particular, Ryan-Smith's construction of models without extendible choice is built using a symmetric extension, and allows us to vary $\EC_\kappa$ at many $\kappa$ independently. 

\begin{definition}
    We say $\EC_\alpha$ holds if when $X = \{A_\beta \mid \beta < \alpha \}$ has choice functions for all `initial segments' $\{A_\beta \mid \beta < \lambda \}$, then $X$ has a choice function.
\end{definition}

For further details relating to extendible choice, see the appendix. 

\begin{theorem}[See Theorem \ref{extendible:theorem}]
     Assume $\GCH$ and let $\mathcal{C}$ be a class of infinite regular cardinals such that whenever $\kappa$ is inaccessible and $\kappa \cap \mathcal{C}$ is unbounded below $\kappa$, we have $\kappa \in \mathcal{C}$. Then there is a symmetric extension $M$ such that $M \models (\forall \kappa) \EC_\kappa \leftrightarrow \text{cf}(\kappa) \notin \mathcal{C}$.
 \end{theorem}

However, we need our system of buttons and switches to also be independent from each other, and this is what causes problems when we try the naive construction. Let us define a system of buttons and switches, where $b_n := \text{`the set }S_n\text{ is not stationary'}$, where $\sqcup S_\alpha$ is a countable disjoint partition of $\omega_1$ and $s_n := \text{`}\EC_\kappa$ holds at the $n$th regular cardinal, for $n \geqslant 2$. We will call this system $(\star)$. 

\begin{prop}
    The system $(\star)$ of buttons and switches is not independent. 
\end{prop}

\begin{proof}
    By Theorem \ref{not-fodor}, there exists a symmetric extension $M$ in which the $(\omega, \omega_1)$-Fodor property fails, and by Proposition \ref{not-collapsed}, $\omega_1$ has not been collapsed in $M$. Now, by Proposition \ref{collapsed}, any forcing over $M$ which restores choice will collapse $\omega_1$ to $\omega$, pushing every unpushed button. We necessarily need to be able to restore choice to switch the switches, so this set of buttons and switches cannot be independent from each other.
\end{proof}

In fact, any of our standard examples\footnote{Over a ground model of $\ZFC + \GCH$ below $\aleph_\omega + \aleph_n = \aleph_n^L$ for all $n<\omega$, let $b_n$ be the statement $\aleph_n^L$ is a cardinal and the $L$-least $\aleph^L_n$-Suslin tree $T^L_n$ in $L$ is still $\aleph^L_n$-Suslin. A proof that these are independent buttons, and another similar example of independent buttons are both found in \cite{Friedman-multiverse}. } of an independent system of buttons, which are based on a property of regular cardinals cannot be independent from a system of choice-switches. 

\begin{prop}\label{all-buttons:prop}
    For a button $b$, which states a combinatorial property of a regular cardinal holds, let $V_\alpha$ be the initial segment of the universe which determines whether the button has been pushed. There exists $\kappa > |V_\alpha|^{V[G]}$, which is regular in $V$, and there exists a symmetric system $\mathscr{S} := \langle\p, \mathscr{G}, \mathscr{F}\rangle$ in $V$, such that over $V_\mathscr{S}$, no new sets have been added to $V_\alpha$, and the $(\kappa,\omega)$-Fodor property fails. 
\end{prop}

\begin{proof}
    Let $\mathscr{S}$ be as defined in the proof of Theorem \ref{not-fodor}, the symmetric system which gives a failure of the $(\kappa,\omega)$-Fodor property. The forcing  $\p$ is $\kappa$-distributive, and  by Lemma \ref{4.2}, for $A \subseteq$ Ord, there is some $\alpha<\omega$, such that $A$ has a symmetric $\Q_0 \ast \dot{\Q}_1 \restriction \alpha$-name. Then if $|A| < \kappa$, it must be that $A \in V$, so no new sets are added to $V_\alpha$. 
\end{proof}

\begin{remark}
    The proof as stated required our ground model $V$ to be a model of $\ZFC$. With a small adjustment to the proof, the same result holds assuming the ground model is a model of $\ZF + \SVC$.
\end{remark}

\begin{cor}
    Any forcing which restores choice will collapse $\kappa$ and press the button.
\end{cor}

This immediately gives us our main theorem. 

\begin{main}
    Any independent system of choice-switches is not independent from any known independent system of buttons, that is, buttons which assert some combinatorial property of a regular cardinal holds. 
\end{main}

This reduces Question 1 to the following: 

\begin{question}
    Is there an independent system of buttons which is itself independent from any system of choice-switches? 
\end{question}

Even if our switches do not directly affect choice, it is not clear how inevitable choice is in our constructions. We could use a standard example of an independent set of switches, $\GCH$ holding at the $n$th regular, in a choiceless model. However, if for some initial segment of the universe $V_\alpha$, which is a model of $\ZF$, if $V_\alpha \models \GCH$, then $V_\alpha \models \AC$. It is unclear whether we have examples of switches which do not necessitate the universe being well-orderable in this way.  

\begin{question}
    Is there an independent set of switches $\{s_n \mid n < \omega\}$ such that for every $n$, both $s_n$ and $\neg s_n$ can hold in a model with a non trivial failure of choice? (i.e. the failure of choice is not far above any witness to $s_n$ or $\neg s_n$).
\end{question}

The modal logic of symmetric extensions has thus far been considered in the context of taking symmetric extensions over a model of $\ZFC$, or equivalently, considering models of $\ZF + \SVC$. We can, however, take symmetric extensions over a model of $\ZF$, or specifically over a model of $\ZF + \neg\SVC$. Such models may be defined using a class sized symmetric system. Notable examples of models of $\ZF + \neg \SVC$ include the Gitik model \cite{Gitik}, the Bristol model \cite{karagila-bristol}, the Morris model \cite{karagila-morris}, and `intermediate models with deep failure of choice' \cite{hayut2024intermediatemodelsdeepfailure}. Goldberg also observed that the existence of a Reinhardt cardinal in a model implies that model is not a model of $\SVC$ \cite{Goldberg-Overflow}. It seems likely that the multiverse of symmetric extensions would have a similar structure in this context, but it is not clear whether we know of an independent system of buttons and switches over a model of $\ZF + \neg\SVC$.

\begin{question}\label{notsvc:question}
    What is the modal logic of forcing, when the core of our multiverse is a model of $\ZF + \neg\SVC$? What is the modal logic of symmetric extensions, when the core of our multiverse is a model of $\ZF + \neg\SVC$?
\end{question}

Hamkins and Löwe \cite{Hamkins-modal-logic} studied and asked questions on several topics closely related to the modal logic of forcing. These included the modal logic of forcing over a fixed model of set theory, the modal logic of forcing with parameters, and the modal logic of forcing when we restrict to a particular class of forcing notions. Another further direction for research could be seeing which of their proofs easily generalise to the symmetric extensions case, and asking their open questions in the context of symmetric extensions. 

\section*{Appendix: defining extendible choice}

We will define a weak choice principle, known as extendible choice, that is closely related to the notions of Hartogs and Lindenbaum numbers. 

\begin{definition}
    For a set $X$, we refer to the least ordinal $\alpha$ such that $X$ admits no injection $\alpha \rightarrow M$ as the Hartogs number of $X$, denoted $\aleph(X)$. 
\end{definition}

It easily follows that $\aleph(X)$ must be a cardinal, and that in a model of $\AC$, $\aleph(X) = |X|^+$. In models without choice, it can be useful to think of the \emph{Hartogs number} of a set as measuring `how well-orderable' the set is. 

\begin{definition}
    For a set $X$, its \emph{Lindenbaum number} is the least ordinal $\alpha$ such that there is not a surjection from $|X|$ to $|\alpha|$, denoted $\aleph^*(X)$. 
\end{definition}

It is straightforward to show that $\aleph(X) \leq \aleph^*(X)$, and that under the axiom of choice, $\aleph = \aleph^*$, and in fact the statement $\aleph = \aleph^*$ is equivalent to the axiom of choice for well-ordered families. In models without choice, if for some set $X$, $\aleph(X) < \aleph^*(X)$, we say that $X$ is an \emph{eccentric set}. By using symmetric extensions, we have a lot of control over the possible Hartog--Lindendaum spectrum, (see \cite{karagila-ryan-smith} and \cite{ryan-smith-hartogs} for details).

\begin{definition}
    For a given model $M$ of $\ZF$, the \emph{Hartogs--Lindenbaum spectrum} of $M$, denoted Spec$_\aleph (M)$ is the class \[\text{Spec}_\aleph (M) := \{ \langle \lambda, \kappa \rangle \mid (\exists X) \aleph(X) = \lambda, \aleph^*(X) = \kappa \}.\]
\end{definition}

There is a necessary core to the spectrum in every model of $\ZF$, made up of the $\{\langle \lambda^+, \lambda^+ \rangle \mid \lambda \in \text{Card} \}$, since $\aleph(\lambda) = \aleph^*(\lambda) = \lambda^+$ for all cardinals $\lambda$.

Using this fine control of the existence of eccentric sets, we can witness very particular failures of choice. Extendible choice was first defined by Levy \cite{Levy}, but we will use a more modern presentation by Ryan-Smith \cite{RS25eccentricity}. 

\begin{definition}\label{EC}
    We say $\EC_\alpha$ holds if when $X = \{A_\beta \mid \beta < \alpha \}$ has choice functions for all `initial segments' $\{A_\beta \mid \beta < \lambda \}$, then $X$ has a choice function.
\end{definition}

The following equivalences of $\EC_\alpha$ are shown in \cite{RS25eccentricity}, using previous results of Levy.

\begin{prop}
    Let $\kappa$ be an infinite ordinal and $\mathcal{C}$ be the class of cardinals $\mu$ with cf$(\mu) =$cf$(\kappa)$. The following are equivalent: 
    \begin{enumerate}
        \item $\EC_\kappa$
        \item for all $\mu \in \mathcal{C}, \EC_\mu$
        \item for some $\mu \in \mathcal{C}, \EC_\mu$
        \item there is a limit $\mu \in \mathcal{C}$ and a set $X$ such that $\aleph(X) = \mu$
    \end{enumerate}
\end{prop}

We also have a characterisation of exactly where $\EC$ can hold. 

\begin{prop}\label{ec-at-sucessors}
    For all ordinals $\alpha$ and $\beta$, 
    \begin{enumerate}
        \item $\EC_0$
        \item $\EC_{\alpha + 1}$
        \item if cf$(\alpha) =$cf$(\beta)$ then $EC_\alpha = EC_\beta$. 
    \end{enumerate}
\end{prop}

 That is, $EC_\alpha$ must hold when $\alpha$ is $0$ or a successor ordinal, and may or may not hold at limit ordinals. 

\begin{theorem} [Theorem 4.10 in \cite{RS25eccentricity}.]\label{extendible:theorem}
     Assume $\GCH$ and let $\mathcal{C}$ be a class of infinite regular cardinals such that whenever $\kappa$ is inaccessible and $\kappa \cap \mathcal{C}$ is unbounded below $\kappa$, we have $\kappa \in \mathcal{C}$. Then there is a symmetric extension $M$ such that \[\text{Spec}_\aleph(M) = \{ \langle \kappa^+, \kappa^+ \rangle \mid \kappa \in \text{Card} \} \cup \{\langle \kappa, \kappa^+ \rangle \mid \text{cf}(\kappa) \in \mathcal{C} \}.\]

     In particular, $M \models (\forall \kappa) \EC_\kappa \leftrightarrow \text{cf}(\kappa) \notin \mathcal{C}$.
 \end{theorem}

\section*{Acknowledgments}

The author would like to thank Andrew Brooke-Taylor, Azul Fatalini and Asaf Karagila for feedback on an early drafts of this paper. The author would also like Fernando Barrera and Benedikt Löwe for useful conversations about modal logic, and Calliope Ryan-Smith for useful comments on Extendible Choice.

\bibliographystyle{plain}
\bibliography{references}

\end{document}